\documentclass[11pt]{article} 

\usepackage[T1]{fontenc}
\usepackage[utf8]{inputenc}
\usepackage{amsmath}
\usepackage{amssymb}
\usepackage{amsthm}
\usepackage{amsfonts}
\usepackage{bbm}
\usepackage[hidelinks]{hyperref}
\usepackage{lmodern}
\usepackage{microtype}
\usepackage[a4paper,margin=3cm]{geometry}
\usepackage{graphicx}
\newtheorem{definition}{Definition}
\newtheorem{theorem}[definition]{Theorem}

\newtheorem{lemma}[definition]{Lemma}

\theoremstyle{definition}

\DeclareMathOperator{\idop}{\mathbbm{1}}

\DeclareMathOperator{\ZZ}{\mathbb{Z}}

\DeclareMathOperator{\RR}{\mathbb{R}}
\DeclareMathOperator{\CC}{\mathbb{C}}

\DeclareMathOperator{\EE}{\mathbb{E}}

\DeclareMathOperator{\GS}{\mathfrak{S}}
\DeclareMathOperator{\ee}{\mathrm{e}}
\newcommand*\diff{\mathop{}\!\mathrm{d}}

\newcommand{\stirlingone}[2]{\genfrac{[}{]}{0pt}{}{#1}{#2}}
\newcommand{\stirlingtwo}[2]{\genfrac{\lbrace}{\rbrace}{0pt}{}{#1}{#2}}
\renewcommand{\phi}{\varphi}

\title{Some observations on the connection between Stirling numbers and Bessel numbers}

\author{David Stenlund \\[0.5em]
	{\small \AA bo Akademi University, Turku, Finland} \\
	{\small \ttfamily david.stenlund@abo.fi}
}

\date{\today}

\begin{document}

\maketitle

\begin{abstract}
We present new proofs for some summation identities involving Stirling numbers of both first and second kind. The two main identities show a connection between Stirling numbers and Bessel numbers. Our method is based on solving a particular recurrence relation in two different ways and comparing the coefficients in the resulting polynomial expressions. We also briefly discuss a probabilistic setting where this recurrence relation occurs. 
\vspace{2ex}
\vspace{2ex}
\end{abstract}

\section{Introduction}

Stirling numbers of the first and second kind are well-known numbers that are found in numerous combinatorial problems. In this short paper, we consider two identities that connect a sum containing both kinds of Stirling numbers with either the first or the second kind of the Bessel numbers. We present a new proof for these identities, together with some remarks on where one of them is found in the context of the positive occupation time for a skew Brownian motion. 

\emph{Unsigned Stirling numbers of the first kind}~$\stirlingone{n}{k}$ and \emph{Stirling numbers of the second kind}~$\stirlingtwo{n}{k}$ are defined recursively for all $n,k\in\ZZ$ through
\begin{equation*}
\stirlingone{n+1}{k} = n \stirlingone{n}{k} + \stirlingone{n}{k-1} \quad \text{and} \quad \stirlingtwo{n+1}{k} = k \stirlingtwo{n}{k} + \stirlingtwo{n}{k-1},
\end{equation*}
with initial conditions
\begin{equation*}
\stirlingone{n}{0} = \stirlingtwo{n}{0} = \delta_{n,0}, \,\quad \stirlingone{0}{k} = \stirlingtwo{0}{k} = \delta_{0,k},
\end{equation*}
where $\delta$ is the Kronecker delta. A combinatorial interpretation of these numbers is that $\stirlingone{n}{k}$ counts the number of permutations of $n$ elements with $k$ disjoint cycles, while $\stirlingtwo{n}{k}$ corresponds to the number of ways to partition a set of $n$ elements into $k$ nonempty subsets. The numbers $(-1)^{n-k}\stirlingone{n}{k}$ are called (signed) Stirling numbers of the first kind. 

Another way of defining the unsigned Stirling numbers of the first kind is that they are the coefficients of a rising factorial,~\cite[Eq.~(6.11)]{GrahamKnuthPatashnik1994}
\begin{equation}\label{eq_s1_rising}
\frac{\Gamma(x+n)}{\Gamma(x)} = \sum_{k=0}^{n} \stirlingone{n}{k} x^k. 
\end{equation}
Similarly, Stirling numbers of the second kind are the coefficients when ordinary powers are expressed using falling factorials,~\cite[Eq.~(6.10)]{GrahamKnuthPatashnik1994}
\begin{equation*}
x^n = \sum_{k=0}^{n} \stirlingtwo{n}{k} \frac{\Gamma(x+1)}{\Gamma(x-k+1)}.
\end{equation*}

The notation for Stirling numbers varies. A rather common notation is $s(n,k)$ and $S(n,k)$ for Stirling numbers of the first and second kind, respectively. This notation is used, for instance, by Comtet~\cite{Comtet1974} and Mansour and Schork~\cite{MansourSchork2016}. However, here we follow the notation used by Graham, Knuth and Patashnik~\cite{GrahamKnuthPatashnik1994} (see also Knuth's remarks on the subject~\cite{Knuth1992}), and refer to their book for basic properties of Stirling numbers, including a number of useful summation identities. 

The \emph{Bessel polynomials} are a sequence of polynomials that are solutions to the second order differential equation
\begin{equation*}
x^2 y_n''(x) + (2x+2) y_n'(x) - n(n+1)y_n(x) = 0,
\end{equation*}
with the normalization $y_n(0)=1$. The first few Bessel polynomials are
\begin{align*}
y_0(x) &= 1, \\
y_1(x) &= x+1, \\
y_2(x) &= 3x^2+3x+1, \\
y_3(x) &= 15x^3+15x^2+6x+1, \\
y_4(x) &= 105x^4+105x^3+45x^2+10x+1.
\end{align*}
Sometimes the polynomials with coefficients in opposite order are used instead. These are consequently called \emph{reverse Bessel polynomials}, and are given by
\begin{equation*}
\theta_n(x) := x^n y_n(1/x).
\end{equation*}
Both these types of Bessel polynomials may also be defined recursively, namely
\begin{equation}\label{eq_Bpolyrec}
\begin{gathered}
y_0(x) = 1, \quad y_1(x) = x+1, \quad y_n(x) = (2n-1)x\,y_{n-1}(x) + y_{n-2}(x), \\[0.1em]
\,\theta_0(x) = 1, \quad \theta_1(x) = x+1, \quad \theta_n(x) = (2n-1)\theta_{n-1}(x) + x^2 \theta_{n-2}(x).
\end{gathered}
\end{equation}
The name Bessel polynomials was first used by Krall and Frink~\cite{KrallFrink1949} due to the connection to Bessel functions. Indeed, the Bessel polynomials can also be written
\begin{equation*}
y_n(x) = \sqrt{\frac{2}{\pi x}}\ee^{1/x} K_{n+\frac{1}{2}}(1/x), 
\end{equation*}
where~$K_\nu(x)$ is a modified Bessel function of the second kind. We refer to Grosswald~\cite{Grosswald1978} for more properties of Bessel polynomials. 

The (signed) \emph{Bessel number of the first kind} $b(n,k)$ is defined as the coefficient before $x^{n-k}$ in the polynomial $y_{n-1}(-x)$.
Explicitly, these numbers are given for $1\leq k\leq n$ by
\begin{equation}\label{eq_defbex}
b(n,k) = (-1)^{n-k}\frac{(2n-k-1)!}{2^{n-k} (k-1)!(n-k)!}. 
\end{equation}
The \emph{Bessel numbers of the second kind} $B(n,k)$ are given for $\lceil \frac{n}{2} \rceil \leq k \leq n$ by
\begin{equation}\label{eq_defbex2}
B(n,k) = \frac{n!}{2^{n-k} (2k-n)!(n-k)!}. 
\end{equation}
These numbers are related through $(-1)^{n-k} B(n,k) = b(k+1,2k-n+1)$, and they form a dual pair in the sense that~\cite[Eq.~(11)--(12)]{YangQiao2011}
\begin{equation}\label{eq_dualBessel}
\sum_{k=m}^n B(n,k) b(k,m) = \sum_{k=m}^n b(n,k) B(k,m) = \delta_{m,n}
\end{equation}
for any $m,n\in\ZZ_+$. A similar duality will be shown also for Stirling numbers, see~\eqref{eq_inversion}. 

The main topic of this paper is sums of the form
\begin{equation*}
\sum_{k=m}^{n} \stirlingone{n}{k}\stirlingtwo{k}{m} z^{k},
\end{equation*}
for certain special values of $z\in\RR$. It has earlier been shown that the general solution can be expressed using \emph{degenerate Stirling numbers}~\cite[Eq.~(24)--(25)]{YangQiao2011} or \emph{generalized Stirling numbers}~\cite[Prop.~8.129]{MansourSchork2016} (see also~\eqref{eq_SSS2} in Section~\ref{section_prob} below). These are two examples of extended versions of the Stirling numbers, and in~\cite{MansourSchork2016} are listed several more variations. Here, however, we will for the most part not consider these generalized versions, but rather use the ordinary Stirling numbers. 

Of particular interest are the following two results, which provide a link between Stirling numbers (of both kinds) and Bessel numbers of either the first or the second kind. The identities are proven by Yang and Qiao in~\cite{YangQiao2011} using exponential Riordan arrays, and in the next section we present an alternative proof. 

\begin{theorem}\label{thm_bss}
For any $n,k\in\ZZ_+$, 
\begin{equation}\label{eq_bss}
\sum_{i=k}^{n} \stirlingone{n}{i}\stirlingtwo{i}{k}(-2)^{n-i} = b(n,k).
\end{equation}
\end{theorem}

\begin{theorem}\label{thm_bss2}
For any $n,k\in\ZZ_+$, 
\begin{equation}\label{eq_bss2}
\sum_{i=k}^{n} \stirlingone{n}{i}\stirlingtwo{i}{k} (-2)^{i-k} = (-1)^{n-k} B(n,k).
\end{equation}
\end{theorem}

We remark that the identities in Theorems~\ref{thm_bss} and~\ref{thm_bss2} are also special cases of the formula \cite[Prop.~16]{MansourSchorkShattuck2012}
\begin{equation}\label{eq_SSS}
\GS_{\frac{s}{\nu};\nu}(n,k) = \sum_{i=k}^{n} \GS_{\frac{s}{\nu-\sigma};\nu-\sigma}(n,i) \GS_{\frac{s+\sigma-\nu}{\sigma};\sigma}(i,k),
\end{equation}
valid for $s\in\RR, \nu\neq0, \sigma>0$, where $\GS_{s,h}(n,k)$ are generalized Stirling numbers defined recursively for $s\in\RR$ and $h\in\CC\setminus \{0\}$ through
\begin{equation}\label{eq_defGS}
\begin{gathered}
\GS_{s;h}(n+1,k) = \GS_{s;h}(n,k-1) + h(k+s(n-k)) \GS_{s;h}(n,k), \\
\GS_{s;h}(n,0) = \delta_{n,0}, \quad \GS_{s;h}(0,k) = \delta_{0,k}.
\end{gathered}
\end{equation}
For more details on generalized Stirling numbers, as well as further references, see Mansour and Schork~\cite{MansourSchork2016}. Comparing the more general recursion above to those for Stirling numbers, we note that $\GS_{1;1}(n,k) = \stirlingone{n}{k}$ and $\GS_{0;1}(n,k) = \stirlingtwo{n}{k}$. Bessel numbers are also special cases of generalized Stirling numbers, namely $\GS_{2;-1}(n,k) = b(n,k)$ and $\GS_{-1;1}(n,k) = B(n,k)$. Using this, and the fact that
\begin{equation*}
\GS_{s;ah}(n,k) = a^{n-k}\GS_{s;h}(n,k),
\end{equation*}
we see that insertion of $(s, \nu, \sigma) = (-2, -1, 1)$ into \eqref{eq_SSS} gives precisely \eqref{eq_bss}, while insertion of $(s, \nu, \sigma) = (-\frac{1}{2}, \frac{1}{2}, 1)$ gives \eqref{eq_bss2}. 

In this paper we present an alternative proof method for the identities in Theorems~\ref{thm_bss} and~\ref{thm_bss2}, which was found independently and initially unknowingly of these previously known results. The identities are proved by showing that both the left hand side and the right hand side, when multiplied by the same factor, are equal to the coefficients before $x^k$ in a polynomial that solves a particular recurrence relation. The proofs are given in Section~\ref{section_proof} and are of a rather elementary nature, although there is some algebra involved. In Section~\ref{section_ex} the same method is used to derive two other known summation identities for Stirling numbers, and a few other possibilities are suggested. 

An additional aim of this paper is to highlight a probabilistic setting where Theorem~\ref{thm_bss} can be applied. The recurrence relation that is used as basis for the proof is in fact a recursive formula for the moments of the occupation time on $[0,\infty)$ for a skew Bessel process.  In Section~\ref{section_prob} we briefly describe this occupation time and explain the interpretation of the result in terms of skew Bessel processes and the special case of skew Brownian motion.

\section{Proof of the main identities}\label{section_proof}

Consider the following polynomial recurrence relation for $n\geq1$ and $x,z\in\RR$:
\begin{equation}\label{eq_besselrec}
P_1(x,z)=x, \quad P_{n+1}(x,z) = x \binom{n+z}{n} - x \sum_{m=1}^{n} \binom{n-m+z}{n-m+1} P_m(x,z).
\end{equation}
The reason behind studying this particular recurrence relation is explained in Section~\ref{section_prob}. Salminen and Stenlund~\cite{SalminenStenlund2020} prove the following solution to the recurrence relation in \eqref{eq_besselrec}. 
\begin{lemma}\label{lemma_bessel}
For any $n\geq 1$,
\begin{equation}\label{eq_besselmom}
P_{n}(x,z) = \sum_{k=1}^{n}\sum_{j=1}^{k} \frac{ (-1)^{j-1} (j-1)!}{(n-1)!} \stirlingone{n}{k} \stirlingtwo{k}{j} x^{j} z^{k-1}
\end{equation}
is a solution to \eqref{eq_besselrec}.
\end{lemma}
\begin{proof}[Proof (sketch)]
For a more detailed proof we refer to \cite[Thm.~4]{SalminenStenlund2020}. Note that neither of the assumptions $x\in(0,1)$ and $z\in(-1,0)$ that are implicit therein (due to the setting of skew Bessel processes) is necessary for this proof, and thus the result holds without these restrictions on $x$ and $z$. The case $z=0$ is trivial, since $P_{n}(x,0)=x$ for all $n$, and thus we can hereafter assume that $z\neq0$. 
Recall that the result is proved by induction. First, the binomial coefficient inside the sum in~\eqref{eq_besselrec} is rewritten as 
\begin{equation*}
\binom{n-m+z}{n-m+1} = \frac{1}{(n-m+1)!} \sum_{k=m}^{n} \stirlingone{n-m+1}{n-k+1} z^{n-k+1}, 
\end{equation*}
based on~\eqref{eq_s1_rising}. Thereafter, the induction assumption is applied, by inserting the expression in~\eqref{eq_besselmom} for~$P_{m}(x,z)$ in~\eqref{eq_besselrec}, which yields a quadruple sum.  
After changing the summation order, this expression is simplified in a suitable way until the desired form as in \eqref{eq_besselmom} is obtained for $P_{n+1}(x,z)$, and the result follows by induction. Key identities for the simplification steps are
\begin{equation*}
\stirlingone{n+1}{n-j+i+1} \binom{n-j+i}{i-1} = \sum_{k=i}^{j} \stirlingone{k}{i} \stirlingone{n-k+1}{n-j+1} \binom{n}{k-1} 
\end{equation*}
and 
\begin{equation*}
\sum_{i=j}^{k} \stirlingtwo{i}{j} \binom{k}{i-1} = j \stirlingtwo{k+1}{j+1},
\end{equation*}
which are modified versions of equations (6.29) and (6.15) in \cite{GrahamKnuthPatashnik1994}, respectively. 
\end{proof}

The solution to the recurrence relation in \eqref{eq_besselrec} is a polynomial in $x$ and $z$ with Stirling numbers of both first and second kind in the coefficients. For certain fixed values of~$z$, the relation \eqref{eq_besselrec} can be solved in a different way than in the proof of Lemma~\ref{lemma_bessel} to obtain another polynomial expression for the solution. Equating the coefficients in these alternative solutions then leads to a summation identity for Stirling numbers of both kinds. In particular, Theorems~\ref{thm_bss} and~\ref{thm_bss2} can be proved this way. 

\begin{proof}[Proof of Theorem~\ref{thm_bss}]
For this proof we consider the case when $z=-1/2$ (which is also of particular interest due to the probabilistic interpretation presented in Section~\ref{section_prob}). Inserting this value into \eqref{eq_besselrec} and recalling that
\begin{equation*}
\binom{n-\frac{1}{2}}{n} = 2^{-2n}\binom{2n}{n},
\end{equation*}
we get the recurrence relation
\begin{align}\label{eq_skewrec}
P_{n+1}\bigl(x,-\tfrac{1}{2}\bigr) &= \frac{x}{2^{2n}}\binom{2n}{n} + \sum_{k=1}^{n} \frac{x}{2^{2n-2k+1}(n-k+1)} \binom{2n-2k}{n-k} P_{k}\bigl(x,-\tfrac{1}{2}\bigr) \nonumber\\
&= \frac{x}{2^{2n}}\binom{2n}{n} + \sum_{k=0}^{n-1} \frac{x}{2^{2k+1}(k+1)} \binom{2k}{k} P_{n-k}\bigl(x,-\tfrac{1}{2}\bigr).
\end{align}
As this is a special case of \eqref{eq_besselrec}, a solution is given by~\eqref{eq_besselmom} with $z=-1/2$. However, using another proof by induction, we show that for $n\geq 1$, the polynomial
\begin{equation}\label{eq_skewmom}
P_{n}\bigl(x,-\tfrac{1}{2}\bigr) = \sum_{k=0}^{n-1} \binom{n+k-1}{k} \frac{x^{n-k}}{2^{n+k-1}}
\end{equation}
solves \eqref{eq_skewrec}. Clearly, the identity holds for~$n=1$. Assume now that~\eqref{eq_skewmom} holds for all~$n\in\{1,2,\dotsc,N\}$. 
It then follows from~\eqref{eq_skewrec} that 
\begin{align*}
P_{N+1}\bigl(x,-\tfrac{1}{2}\bigr)  
&= \frac{x}{2^{2N}}\binom{2N}{N} + \sum_{i=0}^{N-1} \frac{x}{2^{2i+1}(i+1)} \binom{2i}{i} \sum_{k=0}^{N-i-1} \binom{N-i+k-1}{k} \frac{x^{N-i-k}}{2^{N-i+k-1}} \\
&= \frac{x}{2^{2N}}\binom{2N}{N} + \sum_{i=0}^{N-1}\sum_{k=i}^{N-1} \frac{x^{N+1-k}}{2^{N+k}(i+1)} \binom{2i}{i} \binom{N-2i+k-1}{k-i} \\
&= \frac{x}{2^{2N}}\binom{2N}{N} + \sum_{k=0}^{N-1} \frac{x^{N+1-k}}{2^{N+k}} \sum_{i=0}^{k} \frac{1}{2i+1}\binom{2i+1}{i} \binom{N+k-1-2i}{k-i},
\end{align*}
and the inner sum can be evaluated using the Hagen--Rothe identity
\begin{equation*}
\sum_{k=0}^{n} \frac{a}{a+bk} \binom{a+bk}{k} \binom{c-bk}{n-k} = \binom{a+c}{n}, 
\end{equation*}
which is a special case of Equation~(3.146) in \cite{Gould1972} (see also \cite{Chu2010}). This yields
\begin{align*}
P_{N+1}\bigl(x,-\tfrac{1}{2}\bigr) &= \frac{x}{2^{2N}}\binom{2N}{N} + \sum_{k=0}^{N-1} \binom{N+k}{k} \frac{x^{N+1-k}}{2^{N+k}} \\
&= \sum_{k=0}^{N} \binom{N+k}{k} \frac{x^{N+1-k}}{2^{N+k}},
\end{align*}
and thus, by induction, \eqref{eq_skewmom}~holds for all~$n\geq 1$. On the other hand, inserting $z=-1/2$ into \eqref{eq_besselmom} gives that
\begin{align}\label{eq_besselsol}
P_{n}\bigl(x,-\tfrac{1}{2}\bigr) 
&= \sum_{k=1}^{n} x^{k}\frac{ (-1)^{k-1} (k-1)! }{(n-1)!} \sum_{i=k}^{n} \stirlingone{n}{i} \stirlingtwo{i}{k} ( -2)^{1-i}. 
\end{align}
Equating the coefficients before $x^k$ in \eqref{eq_skewmom} and \eqref{eq_besselsol} we obtain 
\begin{align*}
\frac{ (-1)^{k-1} (k-1)!}{(n-1)!} \sum_{i=k}^{n} \stirlingone{n}{i} \stirlingtwo{i}{k} ( -2)^{1-i} &= \frac{1}{2^{2n-k-1}} \binom{2n-k-1}{n-k}, \quad 1 \leq k \leq n,
\end{align*}
which is equivalent to \eqref{eq_bss}, and the proof is complete. 
\end{proof}

\begin{proof}[Proof of Theorem~\ref{thm_bss2}]
In this case we let $z=-2$. Inserting $n=1$ into \eqref{eq_besselrec} we see that
\begin{equation*}
P_2(x,-2) = - x -x(-2 P_1(x,-2)) = 2x^2-x,
\end{equation*}
which becomes our second initial condition, together with $P_1(x,-2)=x$. For any $n\geq3$ the first term disappears altogether, and only two terms are left in the sum. The recurrence relation thus becomes
\begin{equation}\label{eq_Brec}
P_{n+1}(x,-2) = x\bigl(2P_{n}(x,-2)-P_{n-1}(x,-2)\bigr). 
\end{equation}
This recurrence relation can be solved in a standard way. The characteristic equation
\begin{equation*}
1-2u+\frac{u^2}{x} = 0
\end{equation*}
has the roots $u=x\pm \sqrt{x(x-1)}$. Note that we are only using the variable $x$ for comparing the polynomial coefficients, so we can assume here that $x>1$. A solution to the recurrence relation is then 
\begin{equation*}
P_{n}(x,-2) = c_1 \Bigl(x+ \sqrt{x(x-1)}\Bigr)^n + c_2 \Bigl(x- \sqrt{x(x-1)}\Bigr)^n,
\end{equation*}
and from the initial conditions it is seen that $c_1=c_2=\frac{1}{2}$. Expanding both binomial expressions as sums and combining them we obtain
\begin{align}
P_{n}(x,-2) &= \sum_{i=0}^n \binom{n}{i} x^{n-i} \bigl(x(x-1)\bigr)^\frac{i}{2} \left(\frac{1+(-1)^i}{2}\right) \nonumber\\
&= \sum_{m=0}^{\lfloor \frac{n}{2} \rfloor} \binom{n}{2m} x^{n-m} (x-1)^m \nonumber\\
&= \sum_{m=0}^{\lfloor \frac{n}{2} \rfloor} \sum_{k=0}^{m} (-1)^k \binom{n}{2m} \binom{m}{k} x^{n-k} \nonumber\\
&= \sum_{k=0}^{\lfloor \frac{n}{2} \rfloor} (-1)^k 2^{n-2k-1} \binom{n-k}{k}\frac{n}{n-k} x^{n-k} \nonumber\\
&= \frac{n}{2} \sum_{k=\lceil \frac{n}{2} \rceil}^{n} (-1)^{n-k} \frac{(k-1)! \, 2^{2k-n}}{(n-k)! (2k-n)!} x^{k} , \label{eq_Bsol}
\end{align}
where in the second to last step the identity \cite[Eq. (3.120)]{Gould1972}
\begin{equation*}
\sum_{m=k}^{\lfloor \frac{n}{2} \rfloor} \binom{n}{2m} \binom{m}{k} = 2^{n-2k-1} \binom{n-k}{k}\frac{n}{n-k}
\end{equation*}
is used. An alternative way of obtaining \eqref{eq_Bsol} is to recognize the similarity between the recurrence relation in~\eqref{eq_Brec} and that of the Chebyshev polynomials of the first kind $T_n(x)$. These polynomials can be defined recursively through
\begin{equation*}
T_0(x)=1, \quad T_1(x)=x, \quad T_{n+1}(x) = 2x T_n(x) -T_{n-1}(x). 
\end{equation*}
In fact, we see that $P_{n}(x,-2) = (\sqrt{x})^n T_n(\sqrt{x})$ satisfies \eqref{eq_Brec}, and from the explicit expression for Chebyshev polynomials~\cite[(22.3.6)]{AbramowitzStegun1966} we then get \eqref{eq_Bsol}. 
What remains of the proof is just to compare the coefficients before $x^k$ in~\eqref{eq_besselmom} when $z=-2$ with those in~\eqref{eq_Bsol}. This gives
\begin{equation*}
\frac{ (-1)^{k-1} (k-1)!}{(n-1)!} \sum_{i=k}^{n} \stirlingone{n}{i} \stirlingtwo{i}{k} ( -2)^{i-1} =
\begin{cases}
(-1)^{n-k} \dfrac{n (k-1)! \, 2^{2k-n-1}}{(n-k)! (2k-n)!}, & \lceil \frac{n}{2} \rceil \leq k \leq n, \\
0, & \mathrm{otherwise},
\end{cases}
\end{equation*}
which is equivalent to \eqref{eq_bss2}. 
\end{proof}

\section{Extension to other identities}\label{section_ex}

As we have seen in the previous section, the choice of $z=-1/2$ or $z=-2$ in~\eqref{eq_besselrec} results in an identity connecting Stirling numbers with Bessel numbers of the first or second kind, respectively. Here we consider a couple of other choices of~$z$ that give rise to similar summation identities. The case when $z=-1$ is rather trivial, since 
\begin{equation*}
P_{1}(x,-1) = x, \quad P_{n+1}(x,-1) = x P_{n}(x,-1),
\end{equation*}
and thus $P_{n}(x,-1) = x^n$ for all $n\geq1$. Comparing this to \eqref{eq_besselmom} with $z=-1$ leads to the well-known inversion formula for Stirling numbers, 
\begin{equation}\label{eq_inversion}
\sum_{i=k}^{n} \stirlingone{n}{i}\stirlingtwo{i}{k} (-1)^{n-i} = \delta_{n,k}, 
\end{equation}
which is given, for instance, in \cite[p.~264]{GrahamKnuthPatashnik1994}. Compare this to the corresponding identity~\eqref{eq_dualBessel} for Bessel numbers. 

When $z$ is a positive integer, all of the binomial coefficients in \eqref{eq_besselrec} are also positive integers, and thereby the polynomials~$P_{n}(x,z)$ have integer coefficients. A fairly simple example is $z=1$, which leads to the recurrence relation
\begin{equation*}
P_{1}(x,1) = x, \quad P_{n+1}(x,1) = (n+1) x - x \sum_{k=1}^n P_{k}(x,1). 
\end{equation*}
In this case it is easy to prove by induction that for all $n\geq 1$,
\begin{equation*}
P_{n}(x,1) = 1-(1-x)^n = \sum_{k=1}^n (-1)^{k+1} \binom{n}{k} x^k.
\end{equation*}
When compared to \eqref{eq_besselmom} this yields another known identity, namely
\begin{equation*}
\sum_{i=k}^{n} \stirlingone{n}{i}\stirlingtwo{i}{k} = \frac{(n-1)!}{(k-1)!}\binom{n}{k} =: L(n,k),
\end{equation*}
where $L(n,k)$ are so called \emph{Lah numbers}~\cite[p.~156]{Comtet1974}. 

The method presented here could potentially be used to find identities for other values of $z$ as well. For instance, finding solutions for other positive integer values $\geq  2$ does not seem improbable, although the recurrence relations are not quite as simple as when~$z=1$. 
When $z$ is a negative integer~$-m$, many of the terms in~\eqref{eq_besselrec} disappear. The binomial coefficient in the first term becomes
\begin{equation*}
\binom{n-m}{n} =
\begin{cases}
\displaystyle
(-1)^n \binom{m-1}{n}, & 0 \leq n < m, \\
0, & 0 < m \leq n,
\end{cases}
\end{equation*}
using a well-known formula for negative first index of a binomial coefficient \cite[Eq. (5.14)]{GrahamKnuthPatashnik1994}. Similarly, the binomial coefficient inside the sum in~\eqref{eq_besselrec} is zero whenever $k\leq n-m$. In fact, after the first $m$ steps of the recursion, we are left with a (homogeneous) linear recurrence relation of order $m$. An example of this has already been illustrated in the proof of Theorem~\ref{thm_bss2} with $z=-2$. For negative integer values $z<-2$ the recurrence relations are of higher order,  and thus it is also more difficult in these cases to find the solution and express it as a polynomial in $x$, in order to compare the coefficients. 

Finally, solving the recurrence relation for half integer values of $z$ seems a viable possibility, as it has already been done for~$z=-1/2$ in the proof of Theorem~\ref{thm_bss}. Although a solution has not yet been found for other half integer values, we will certainly investigate these cases further.

\section{A probabilistic interpretation}\label{section_prob}

The significance of the recursive equation \eqref{eq_besselrec} is that it arises in the study of the occupation time on $[0,\infty)$ for skew Bessel diffusion processes. More precisely, let $(X_t)_{t\geq 0}$ be a skew Bessel process~\cite{BarlowPitmanYor1989arcsinus, SalminenStenlund2020, Watanabe1995} starting at $X_0=0$ with skewness parameter $x\in(0,1)$ and parameter $z\in(-1,0)$ related to the dimension ($d=2+2z$), and let further $A_t$ be the occupation time on $[0,\infty)$ of the process up to time $t$, 
\begin{equation*}
A_t := \int_0^t \idop_{[0,\infty)}(X_s) \diff s. 
\end{equation*}
Then, the $n$th moment $\EE(A_1^n)$ of the occupation time up to time~$1$ is equal to~$P_{n}(x,z)$, which is given recursively by~\eqref{eq_besselrec} and explicitly in~\eqref{eq_besselmom}, as shown in \cite[Thms.~3,~4]{SalminenStenlund2020}. 
Note that, due to self-similarity of the Bessel processes, it follows that $\EE(A_t^n)=t^n P_{n}(x,z)$ for all~$t\geq 0$. 

The special case when $z=-1/2$ corresponds to a skew Brownian motion. This is precisely the case considered in the proof of Theorem~\ref{thm_bss}. In fact the proof was attained by first solving the recurrence relation~\eqref{eq_skewrec} for a skew Brownian motion and later using another path to solve the more general recurrence for skew Bessel processes. Comparing the results then led to the interesting identity \eqref{eq_bss}. After some research this identity was found in literature \cite{MansourSchorkShattuck2012, YangQiao2011}, although therein proved using different approaches. 

The moments of $A_1$ for a skew Brownian motion with skewness parameter $x\in(0,1)$ is thus given by~\eqref{eq_skewmom}, which alternatively can be expressed using Bessel numbers as
\begin{equation}\label{eq_sBMbess}
\EE(A_1^n) = \frac{1}{2^{n-1} (n-1)!} \sum_{k=1}^{n} (-1)^{n-k} (k-1)! \, x^{k} \, b(n,k). 
\end{equation}
The sum somewhat resembles a modified version of an inverse Bessel polynomial, since 
\begin{equation*}
x \,\theta_{n-1}(x) = \sum_{k=1}^n (-1)^{n-k} x^{k} \, b(n,k).
\end{equation*}
However, the recurrence for $\EE(A_1^n)$ as given in \eqref{eq_skewrec} is naturally very different from the simple recurrence for inverse Bessel polynomials in \eqref{eq_Bpolyrec}. 

As noted in the introduction, Theorem~\ref{thm_bss} is also a special case of \eqref{eq_SSS}. If we let $\sigma=1$ and $\nu=s+1$ this equation becomes
\begin{equation*}
\GS_{\frac{s}{s+1};s+1}(n,k) = \sum_{i=k}^{n} \GS_{1;s}(n,i) \GS_{0;1}(i,k) = \sum_{i=k}^{n} s^{n-i}\stirlingone{n}{i} \stirlingtwo{i}{k},
\end{equation*}
and, writing $z=s^{-1}$, we obtain 
\begin{equation}\label{eq_SSS2}
\sum_{i=k}^{n} \stirlingone{n}{i} \stirlingtwo{i}{k}z^{i} = z^n \GS_{\frac{1}{z+1};\frac{z+1}{z}}(n,k).
\end{equation}
Thus, for any skew Bessel process with parameter $z\in(-1,0)$, the moments of the positive occupation time can alternatively be expressed using generalized Stirling numbers as
\begin{align*}
\EE(A_1^n) &= \sum_{k=1}^{n} \frac{ (-1)^{k-1} (k-1)! \, x^{k}}{x(n-1)!} \sum_{i=k}^{n} \stirlingone{n}{i} \stirlingtwo{i}{k} z^{i} \\
&= \frac{z^{n-1}}{(n-1)!} \sum_{k=1}^{n} (-1)^{k-1} (k-1)! \, x^{k} \GS_{\frac{1}{z+1};\frac{z+1}{z}}(n,k),
\end{align*}
which is of a similar form as \eqref{eq_sBMbess}. This illustrates how Bessel numbers occur in the moments of occupation times on $[0,\infty)$ for a skew Brownian motion, while generalized Stirling numbers are, correspondingly, found in the more general case of skew Bessel processes.

\subsection*{Acknowledgments} 
The research of D.~Stenlund was supported by a grant from the Magnus Ehrnrooth foundation.


\newpage
\begin{thebibliography}{10}

\bibitem{AbramowitzStegun1966}
Abramowitz, M. and Stegun, I.~A. (1966).
\newblock {\em Handbook of mathematical functions, with formulas, graphs, and
  mathematical tables}.
\newblock Dover Publications, Inc., New York. 

\bibitem{BarlowPitmanYor1989arcsinus}
Barlow, M., Pitman, J. and Yor, M. (1989).
\newblock Une extension multidimensionnelle de la loi de l'arc sinus.
\newblock In: {\em S\'{e}minaire de {P}robabilit\'{e}s, {XXIII}}.
\newblock {\em Lecture Notes in Math.}, Vol.~1372, pp.~294--314. Springer, Berlin. 

\bibitem{Chu2010}
Chu, W. (2010).
\newblock Elementary proofs for convolution identities of {A}bel and
  {H}agen-{R}othe.
\newblock {\em Electron. J. Combin.\/} {\bf 17,} Note 24, 5. 

\bibitem{Comtet1974}
Comtet, L. (1974).
\newblock {\em Advanced combinatorics} 
\newblock D. Reidel Publishing Company, Dordrecht. 

\bibitem{Gould1972}
Gould, H.~W. (1972).
\newblock {\em Combinatorial identities}.
\newblock Henry W. Gould, Morgantown. 

\bibitem{GrahamKnuthPatashnik1994}
Graham, R.~L., Knuth, D.~E. and Patashnik, O. (1994).
\newblock {\em Concrete mathematics}, 2nd~ed.
\newblock Addison-Wesley Publishing Company, Reading.

\bibitem{Grosswald1978}
Grosswald, E. (1978).
\newblock {\em Bessel polynomials}. {\em Lecture Notes in
  Math.}, Vol.~698. 
\newblock Springer, Berlin. 

\bibitem{Knuth1992}
Knuth, D.~E. (1992).
\newblock Two notes on notation.
\newblock {\em Amer. Math. Monthly\/} {\bf 99,} 403--422. 

\bibitem{KrallFrink1949}
Krall, H.~L. and Frink, O. (1949).
\newblock A new class of orthogonal polynomials: {T}he {B}essel polynomials.
\newblock {\em Trans. Amer. Math. Soc.\/} {\bf 65,} 100--115. 

\bibitem{MansourSchork2016}
Mansour, T. and Schork, M. (2016).
\newblock {\em Commutation relations, normal ordering, and {S}tirling numbers}.
\newblock 
CRC Press, Boca Raton. 

\bibitem{MansourSchorkShattuck2012}
Mansour, T., Schork, M. and Shattuck, M. (2012).
\newblock The generalized {S}tirling and {B}ell numbers revisited.
\newblock {\em J. Integer Seq.\/} {\bf 15,} Article 12.8.3, 47. 

\bibitem{SalminenStenlund2020}
Salminen, P. and Stenlund, D. (2020).
\newblock On occupation times of one-dimensional diffusions.
\newblock {\em J. Theor. Probab.\/}
\newblock https://doi.org/10.1007/s10959-020-00993-3

\bibitem{Watanabe1995}
Watanabe, S. (1995).
\newblock Generalized arc-sine laws for one-dimensional diffusion processes and
  random walks.
\newblock In: Cranston, M.~C. and Pinsky, M.~A. (eds.) {\em Stochastic analysis ({I}thaca, {NY}, 1993)}.
\newblock {\em Proc. Sympos. Pure Math.} Vol.~57, pp.~157--172. Amer. Math. Soc.,
  Providence. 

\bibitem{YangQiao2011}
Yang, S.~L. and Qiao, Z.~K. (2011).
\newblock The {B}essel numbers and {B}essel matrices.
\newblock {\em J. Math. Res. Exposition\/} {\bf 31,} 627--636. 

\end{thebibliography}
\end{document}